\documentclass[a4paper]{article}

\usepackage{a4wide}
\usepackage{amsmath,amssymb,amsthm}
\usepackage{mathrsfs,bbm}
\usepackage[american]{babel}
\usepackage[dvips]{graphicx}
\usepackage{array,pifont,varioref,framed}

\graphicspath{{./}{figure/}}

\title{Parameter estimation of monomial-exponential sums}
\author{Luisa Fermo, Cornelis Van der Mee and Sebastiano Seatzu \\
		{\small\it Department of Mathematics and Computer Science} \\
		{\small\it University of Cagliari} \\
		{\small\it Viale Merello 92, 09123 Cagliari, Italy} \\[5mm]
	   }
\date{}

\newcommand{\RR}{\mathbb{R}}

\theoremstyle{plain}\newtheorem{theorem}{Theorem}[section]
\theoremstyle{remark}
\theoremstyle{plain}\newtheorem{lemma}[theorem]{Lemma}
\theoremstyle{plain}
\theoremstyle{definition}

\begin{document}
\maketitle

\begin{abstract}
We propose a numerical method, based upon matrix-pencils, for the identification of parameters and coefficients of a monomial-exponential sum. We note that this method can be considered an extension of the numerical methods for the parameter estimation of exponential sums. The application of the method is applied to several examples, some already present in the literature and others, to our knowledge, never considered before.

\medskip


\medskip

\end{abstract}
\section{Introduction}
Denoting by  $n$ and $\{m_j\}_{j=1}^n$ positive integers, let us consider the following monomial-exponential sum
\begin{equation}\label{hx}
h(x)=\sum_{j=1}^{n} \sum_{s=0}^{m_j-1} c_{js} x^s e^{f_jx},
\end{equation}
where $\{c_{js}\}_{j=1,s=0}^{n,m_j-1}$ and $\{f_j\}_{j=1}^n$ are  complex or real  parameters with $f_j \neq 0$, which reduces to a linear combination of exponentials in the case $m_1=m_2=\dots=m_n=1$. Setting $$M=m_1+m_2+\dots+m_n,$$ we want to recover all  parameters of $h$ given $2N$ ($N \geq M$) observed data.
This problem has many applications in science and engineering. For instance, it arises in  propagation of signals \cite{Papy2005}, electromagnetics \cite{Attiya2011} and high-resolution imaging of moving targets \cite{Hua1994}, as well as in the direct scattering problem concerning the solution of the class of nonlinear partial differential equation of integrable type (see Subsection \ref{applications}).

In the literature there exist several approaches to solve this problem, in the case of exponential sums. The methods used most are  Prony-like (or polynomial) methods and  matrix-pencil methods.
The first ones are based on the paper by  G. de Prony \cite{Prony1795} who 
was the first to investigate this problem. He proposed a quite efficient and accurate approach for extracting parameters under the hypothesis that $n$ is known, $m_j \equiv 1$ and the observed data are exact. This method  was principally based on the solution of two linear systems characterized by a  Hankel and a Vandermonde matrix, respectively. The first system furnishes the coefficients of a polynomial (the so-called Prony polynomial) whose roots allow one to determine the parameters $f_j$, while the second system provides the coefficients $c_{js}$. Several extensions have been proposed (see, for instance, \cite[pp. 458-462]{Hildebrand1956}, \cite{Mittra1975}, \cite{Mittra1978}  \cite{Beylkin2005} and more recently in \cite{Potts2010} and \cite{Potts2011}) to apply this polynomial method, also in the case where $n$ is only approximately known or $m_j \neq 1$ or the data are affected by noise. The matrix-pencil technique has been developed more recently \cite{Hua1990}. As the Prony-like methods, one recovers the coefficients $c_{js}$ by solving a Vandermonde system but (see, for instance, \cite{Pereira1995}) the computation of the parameters $f_j$ is reduced to only one step. In fact, it allows one to estimate the zeros of the Prony polynomial and then $f_j$ without passing through the computation of its coefficients.
This is the main difference with the Prony-like methods, which makes this kind of method more computationally efficient.

More recently, for exponential sums and for noiseless sampled data,  a close connection between the methods mentioned above has been proposed in \cite{Potts2013}, which allows one to obtain a unified approach in the case where an approximate  upper bound $\widehat{M}$ of $n$ is given.  In this context  two algorithms have been proposed \cite{Potts2013}, respectively based on a $QR$ factorization and on the singular value decomposition of a rectangular Hankel matrix. This second technique makes it equivalent to the 
ESPIRIT (Estimation of Signal Parameters via Rotational Invariance Techniques) method (see, for instance, \cite{Roy1990}). 

In this paper we propose a new matrix-pencil method which allows one to solve the problem in the more general case of monomial-exponential sums also in the presence of noisy data and under the hypothesis that we know a reasonable upper bound of $M$.

As usual in the Prony-like methods, first we introduce the Prony polynomial, namely a monic polynomial of degree $M$ having $z_j$ as its $j$th zero with multiplicity $m_j$, and then we  arrange the data in two square Hankel matrices of order $N$.  By using difference equation theory, we state some important properties of these matrices which are basic to our method. We introduce a matrix-pencil and prove that the parameters $f_j$ we are looking for are exactly the generalized eigenvalues of this special matrix which we compute by resorting to the Generalized Singular Value Decomposition \cite{Golub1996}. Finally, we solve an overdetermined system with a Casorati matrix to recover the coefficients $c_{js}$.

The paper is organized as follows. In Section \ref{method1}, we illustrate our method, assuming  $M=m_1+m_2+\dots+m_n$ exactly known. In Section \ref{method2} we explain what changes are needed if we do not know exactly $M$  but only an  upper bound. Section \ref{tests} is devoted to the results of our numerical experimentation, while conclusions follow in Section \ref{conclusions}.
 
\section{The numerical method}\label{method1}
In this section we present the numerical method we propose to recover all parameters appearing in the monomial-exponential sum \eqref{hx}.
More precisely, we reduce the non-linear approximation problem to two  problems of linear algebra. The first one is a generalized eigenvalue problem, which allows us to recover $n$, $z_j$ and $m_j$. The second one is the solution of a linear system with a Casorati matrix to compute the parameters $c_{js}$.

Firstly we note that, 
setting $z_j=e^{f_j} \neq 0$, we can rewrite the monomial exponential sum \eqref{hx}  as a monomial-power sum
\begin{equation}\label{hx1}
h(x)=\sum_{j=1}^{n} \sum_{s=0}^{m_j-1} c_{js} x^s z_j^x.
\end{equation}
Moreover, let $M=n_1+\,\dots\,+n_n$ and  assume that  $2N$ sampled data with $N>M$
\begin{equation}\label{hk1}
h(k)=\sum_{j=1}^{n} \sum_{s=0}^{m_j-1} c_{js} k^s z_j^k, \quad 0^0 \equiv 1 
\end{equation}
are given for the $2N$  values $k=k_0,k_0+1,\dots, k_0+2N-1$ with $k_0 \in \mathbb{N}^+=\{0,1,2,...,k_0,... \}$.
Preliminary, we arrange the $2N$ given data in the following square Hankel matrices of order $N$

\begin{align}\label{Hankel0}
{\bf{H}}_{NN}^{k_0}=
\left(\begin{matrix}
h(k_0) & h(k_0+1) & \dots & h(k_0+N-1) \\
h(k_0+1) & h(k_0+2) & \dots & h(k_0+N) \\
\vdots & \vdots & \vdots & \vdots \\
h(k_0+N-1) & h(k_0+N) & \dots & h(k_0+2N-2) 
\end{matrix}\right)=[\mathbf{h}_{k_0},\mathbf{h}_{k_0+1}, \dots, ,\mathbf{h}_{k_{0}+N-1}] 
\end{align}
\begin{align}\label{Hankel1}
{\bf{H}}_{NN}^{k_0+1}=
\left(\begin{matrix}
h(k_0+1) & h(k_0+2) & \dots & h(k_0+N) \\
h(k_0+2) & h(k_0+3) & \dots & h(k_0+N+1) \\
\vdots & \vdots & \vdots & \vdots \\
h(k_0+N) & h(k_0+N+1) & \dots & h(k_0+2N-1) 
\end{matrix}\right)=[\mathbf{h}_{k_0+1},\mathbf{h}_{k_0+2}, \dots, \mathbf{h}_{k_{0}+N}].
\end{align}

Notice that ${\bf{H}}_{NN}^{k_0+1}$ is essentially a shift of ${\bf{H}}_{NN}^{k_0}$, as the first $N-1$ columns of ${\bf{H}}_{NN}^{k_0+1}$ coincide with the last $N-1$ columns of ${\bf{H}}_{NN}^{k_0}$ apart from the last entry. 

In the following we will often write  ${\bf{H}}_{NM}^{k_0}$ and  ${\bf{H}}_{NM}^{k_0+1}$, each of order $N \times M$ with $N \geq M$, for the truncation  Hankel matrices ${\bf{H}}^{k_0}_{NN}$ and  ${\bf{H}}^{k_0+1}_{NN}$, respectively formed by their first $M$ columns.

The next lemma contains two properties of  these Hankel matrices that are relevant to our method.

\begin{lemma} \label{lemma}
Let us assume $M$  known and the sampled data noiseless.  
Then: 
\begin{itemize}
\item[(a)] The matrices \eqref{Hankel0} and \eqref{Hankel1} have rank $M$, that is 
\begin{equation} \label{rank} 
{\rm{rank}} \, {\bf{H}}_{NN}^{k_0}={\rm{rank}} \, {\bf{H}}_{NN}^{k_0+1}=M;
\end{equation}
\item[(b)] The following relation holds true
\begin{equation} \label{connection}
{\bf{H}}_{NM}^{k_0+1} = {\bf{H}}_{NM}^{k_0} \, {\bf C}_{M}(P)
\end{equation}
where ${\bf C}_M(P)$ is the companion matrix of the Prony polynomial, i.e.
\begin{align*}
{\bf{C}}_{M}(P)=
\left(\begin{matrix}
0 & 0 & \dots & 0 & -p_0 \\
1 & 0 & \dots & 0 & -p_1 \\
\vdots & \vdots & \vdots & \vdots \\
0 & 0 & \dots & 1 & -p_{M-1} 
\end{matrix}\right).
\end{align*}
\end{itemize}
\end{lemma}

\begin{proof}
To prove $(a)$, we interpret $h(k)$ as the general solution of a  homogeneous linear difference equation  of order $M$
\begin{equation}\label{difference}
\sum_{k=0}^{M} p_k h_{k+m}= 0, \quad p_M=1
\end{equation}
whose characteristic polynomial is the Prony polynomial, i.e. the monic polynomial of degree $M$ having $z_j$ as the $j$th zero with multiplicity $m_j$ 
\begin{equation}\label{Prony}
P(z)=\prod_{j=1}^{n} (z-z_j)^{m_j} =\sum_{k=0}^M p_k z^k, \quad p_M \equiv 1.
\end{equation}

It is well known that equation \eqref{difference}, regardless of the values $\{p_k\}_{k=0}^{M-1}$, has a unique solution $h_k$, for each given set of $M$ initial conditions $h_{k_0}, h_{k_0+1}, \dots, h_{k_0+M-1}$ \cite{LakTri2002}.\\
Since \eqref{Prony} is the characteristic polynomial of equation \eqref{difference}, each function
$h_{j,s}(k)=k^s z_j^k, j=1, \dots, n, s=0,1, \dots, m_j-1,$
is a solution of  \eqref{difference}.  Moroever, they are  linearly independent \cite[Theorem 2.2.3]{LakTri2002} and represent a basis for the vector space of solutions of \eqref{difference}.
Hence the function $h(k)$  is the general solution of \eqref{difference} and its coefficients $\{c_{js}\}_{j=1,s=0}^{n,m_j-1}$ can be uniquely determined by fixing $M$ initial values $	h(k_0), h(k_0+1), \cdots,  h(k_0+M-1) $.
Then, if we consider  the first $M$ columns  $\mathbf{h}_0,\mathbf{h}_1, \dots, \mathbf{h}_{M-1}$ of $\mathbf{H}_{NN}^{k_0}$ as initial data, we can say that the columns $\mathbf{h}_M,\mathbf{h}_{M+1}, \dots, \mathbf{h}_{N}$ are a linear combination of the first ones. As a result, ${\rm{rank}} \, {\bf{H}}_{NN}^{k_0}=M$. The same conclusion holds if $k_0$ is replaced by $k_0+1$, so that ${\rm{rank}} \, {\bf{H}}_{NN}^{k_0+1}=M$.

Relation \eqref{connection} is immediate as the product between ${\bf{H}}_{NM}^{k_0}$ and the $j$th column of  ${\bf C}_{M}(P)$ gives the $(j+1)$th column of ${\bf{H}}_{NM}^{k_0+1}$    and further, by virtue of \eqref{difference}, we have
$$ -\sum_{k=0}^{M-1}p_k h_{k+k_0}=h_{k_0+M}.$$
\end{proof}

The next theorem contains two results basic to our method.
\begin{theorem}
The zeros of the Prony polynomial, with their multiplicities, are exactly  
the eigenvalues, with the same multiplicity,  of the matrix-pencil
\begin{equation}
{\bf{H}}_{MM}(z)=({{\bf{H}}_{NM}^{k_0}})^* ({\bf{H}}_{NM}^{k_0+1} - z {\bf{H}}_{NM}^{k_0})
\end{equation} 
where the asterisk denotes the conjugate transpose.

Moreover,  the coefficients $c_{js}$ appearing in  \eqref{hx} are the solutions of the linear system 

\begin{equation}\label{Casorati_system}
\mathbf{K}_M^{k_0} \mathbf{c}=\mathbf{h}^{k_0}
\end{equation}
where $\mathbf{c}=[c_{1,0},...,c_{1,n_1-1},...,c_{M,0},...,c_{M,n_{n}-1}]^T$,   ${\bf{h}}^{k_0}=[h(k_0),\, h(k_0+1),\, \dots,\, h(k_0+M-1)]^T$ and  $\mathbf{K}_M^{k_0}$ is the Casorati matrix
\begin{equation}\label{Casorati_matrix}
\mathbf{K}_M^{k_0}=\left(\begin{matrix}
z_1^{k_0} & k_0 z_1^{k_0} & \dots & k_0^{n_1-1}z_1^{k_0} & \dots & z_n^{k_0} & k_0 z_n^{k_0} & \dots & k_0^{n_n-1} z_n^{k_0} \\
z_1^{k_1} & k_1 z_1^{k_1} & \dots & k_1^{n_1-1}z_1^{k_1} & \dots & z_n^{k_1} & k_1 z_n^{k_1} & \dots & k_1^{n_n-1} z_n^{k_1} \\
\vdots & \vdots & \vdots & \vdots & \vdots & \vdots & \vdots & \vdots & \vdots \\
z_1^{k_{M-1}} & k_{M-1} z_1^{k_{M-1}} & \dots & k_{M-1}^{n_1-1}z_1^{k_{M-1}} & \dots & z_n^{k_{M-1}} & k_{M-1} z_n^{k_{M-1}} & \dots & k_{M-1}^{n_n-1} z_n^{k_{M-1}} \\
\end{matrix}\right).
\end{equation}

\end{theorem}
\begin{proof}
By using \eqref{connection}, we can write
\begin{equation}
{\bf{H}}_{MM}(z)= ({{\bf{H}}_{NM}^{k_0}})^* {\bf{H}}_{NM}^{k_0} ({\bf C}_{M}(P)-z {\bf{I}}_{MM})  
\end{equation}
where ${\bf{I}}_{MM}$ is the identity matrix of order $M$. Hence, 
the first statement follows by  noting that
$$ {\rm{det}} \,{\bf{H}}_{MM}(z)={\rm{det}}(({{\bf{H}}_{NM}^{k_0}})^* {\bf{H}}_{NM}^{k_0})\, {\rm{det}} ({\bf C}_{M}(P)-z {\bf{I}}_{MM})={\rm{det}}(({{\bf{H}}_{NM}^{k_0}})^* {\bf{H}}_{NM}^{k_0})\, P(z),$$
and by taking into account that  ${\rm{det}}(({{\bf{H}}_{NM}^{k_0}})^* {\bf{H}}_{NM}^{k_0}) \neq 0$ as ${\bf{H}}_{NM}^{k_0}$ has full rank.
Concerning system \eqref{Casorati_system}, we note that its matrix is non singular regardless of the $k_0$ value as it is the Casorati matrix, which plays in the theory of difference equations the same role as the Wronskian matrix in the theory of differential equations. Notice that  the Casorati matrix coincides with the Vandermonde matrix ${\bf{V}}_{M}=[z_j^{k_i}]_{i=0,j=1}^{M-1,n}$ whenever all zeros $z_j$ are simples ($\eta_j \equiv 1$). 
\end{proof}

{\emph{Computation of $\{z_j,\, m_j \,, f_j \}$.}} Knowing $M$, the computation of the parameters we are looking for, 
can then be carried out  by solving the following generalized eigenvalue problem
\begin{equation}
 ({{\bf{H}}_{NM}^{k_0}})^* {\bf{H}}_{NM}^{k_0+1} x=z  ({{\bf{H}}_{NM}^{k_0}})^* {\bf{H}}_{NM}^{k_0} x, \quad x \neq 0.
\end{equation}

To this end we factorize the matrices  ${\bf{H}}_{NM}^{k_0+1}$ and ${\bf{H}}_{NM}^{k_0}$ by means of the Generalized Singular Value Decomposition (GSVD) \cite{Golub1996} 

\begin{align}
{\bf{H}}_{NM}^{k_0+1}&={\bf{U}}_{NN} \left(\begin{matrix}
{\bf{\Sigma}}_{MM}^{k_0+1}\\
{\bf{0}}_{N-M,M}
\end{matrix}\right) {\bf{X}}_{MM} \label{gsvd1}\\
{\bf{H}}_{NM}^{k_0}&={\bf{V}}_{NN} \left(\begin{matrix}
{\bf{\Sigma}}_{MM}^{k_0}\\
{\bf{0}}_{N-M,M}
\end{matrix}\right) {\bf{X}}_{MM} \label{gsvd2}
\end{align}
where ${\bf{\Sigma}}_{MM}^{k_0+1}$ and ${\bf{\Sigma}}_{MM}^{k_0}$ are two non-negative diagonal matrices of order $M$, ${\bf{U}}_{NN}$ and ${\bf{V}}_{NN}$ are two square unitary matrices of order $M$, ${\bf{X}}_{MM}$ is a nonsingular matrix of order $M$ and $\bf{O}_{N-M,M}$ is the null matrix of order $(N-M) \times M$.

Thus, by using \eqref{gsvd1} and \eqref{gsvd2}, we can rewrite the matrix-pencil as
\begin{align*}
{\bf{H}}_{MM}(z)& =({\bf{X}}_{MM})^* \, \left[\begin{matrix}
({\bf{\Sigma}}_{MM}^{k_0})^* & 
{\bf{0}}_{M,N-M}
\end{matrix}\right] ({\bf{V}}_{NN})^* {\bf{U}}_{NN} \left(\begin{matrix}
{\bf{\Sigma}}_{MM}^{k_0+1}\\
{\bf{0}}_{N-M,M}
\end{matrix}\right) {\bf{X}}_{MM} \\ & \hspace{1 cm}-z ({\bf{X}}_{MM})^* \, \left[\begin{matrix}
({\bf{\Sigma}}_{MM}^{k_0})^* & 
{\bf{0}}_{M,N-M}
\end{matrix}\right]  \left(\begin{matrix}
{\bf{\Sigma}}_{MM}^{k_0}\\
{\bf{0}}_{N-M,M}
\end{matrix}\right) {\bf{X}}_{MM} \\ 
& = 
({\bf{X}}_{MM})^* \, ({\bf{\Sigma}}_{MM}^{k_0})^*  \left[ ({\bf{V}}_{NM})^* {\bf{U}}_{NM} 
{\bf{\Sigma}}_{MM}^{k_0+1}-z {\bf{\Sigma}}_{MM}^{k_0}  \right] {\bf{X}}_{MM}.
\end{align*}

As a result, the generalized eigenvalues of the matrix-pencil, and then the zeros of the Prony polynomial, are exactly the eigenvalues of the matrix 

$$ ({\bf{\Sigma}}_{MM}^{k_0})^{-1} ({\bf{V}}_{NM})^* {\bf{U}}_{NM} 
{\bf{\Sigma}}_{MM}^{k_0+1},$$
which can be effectively computed by using the $\emph{eig}$ algorithm of MATLAB.

In this way we compute the zeros $z_j$ with their multiplicities $m_j$ and of course $n$. The computation of $f_j$ is immediate as $z_j=e^{f_j}$, $j=1 \,, \dots,\,n.$

It is interesting to note that if $N=M$, the zeros $z_j$ of the Prony polynomial can be computed by considering the simple matrix-pencil 
$$\widehat{H}_{MM}(z)=H^{k_0+1}_{MM}-zH^{k_0}_{MM}.$$
In this case, considering that $H^{k_0+1}_{MM}$ and $H^{k_0}_{MM}$ are symmetric the $QZ$ technique \cite{Golub1996} is a very effective technique as explained in \cite{Golub1999}. In this paper, we do not consider this case because our numerical experiments show that using all available data $h(k)$ is more effective, although the numerical procedure is computationally more complex. This numerical evidence agrees with those obtained in the parameter estimation for exponential sums \cite{Potts2013}.

{\emph{Computation of $\{c_{js} \}$.}}
Once $\{n,\, z_j,\, m_j \}$ has been computed, we are in a position to evaluate the coefficients $c_{js}$, given $h(k)$ in $M$ distinct points $\{k_0,\, k_1,\, \dots,\,k_{M-1} \}$. Indeed, we can write down the Casorati matrix and then solve linear system \eqref{Casorati_system}.

Although theoretically not necessary, our numerical tests suggest to use more then $2M$ data. For this reason, whenever it is possible we prefer to use $2N$ ($N>M$) sampled data and to compute the eigenvalues by solving, in the least squares sense, the overdetermined linear system 
\begin{equation}
\mathbf{K}_{2N, M}^{k_0} \mathbf{c}=\mathbf{h}^{k_0}
\end{equation}
where $\mathbf{h}^{k_0}=[h(k_0),\, h(k_0+1),\, \dots,\, h(k_0+2N-1)]$ and $\mathbf{K}_{NM}^{k_0}$ is the Casorati matrix of order $2N \times M$ ($N>M$), obtained as a natural extension of \eqref{Casorati_matrix}. As can be expected, this extention is increasingly important as the ratio noise/signal increases.


\section{Not knowing the value of $M$}\label{method2}

Now we assume that $M$, that is the exact number of terms in \eqref{hx}, is an unknown parameter, assuming that, as usual in applications, only a reasonable upper bound $ \widehat{M}$ of $M$ is known.

Under this hypothesis, we want to recover all of the parameters and coefficients $\{ n,m_j,f_j, c_{js}\}$ of \eqref{hx} assuming to have an estimate of $h(k)$ in a set of $2N$ data $\{k_0,\, k_0+1,\,\dots,\,k_0+2N-1 \} \in \mathbb{N}^+_{k_0}$ with $N \geq \widehat{M}$.
In this case we have first to estimate $M$, which can be done by using the following.
\begin{theorem}
In the absence of  noise on the data, the rank of the $N \times \widehat{M}$ Hankel matrix 

\begin{align*}
\mathbf{H}_{N\widehat{M}}^{k_0}=\left(\begin{matrix}
h(k_0) & h(k_0+1) & \dots & h(k_0+\widehat{M}-1) \\
h(k_0+1) & h(k_0+2) & \dots & h(k_0+\widehat{M}) \\
\vdots & \vdots & \vdots & \vdots \\
h(k_0+N-1) & h(k_0+N) & \dots & h(k_0+N+\widehat{M}-2) 
\end{matrix}\right)=[\mathbf{h}^{k_0},\mathbf{h}^{h_0+1}, \dots, \mathbf{h}^{k_0+\widehat{M}-1}]
\end{align*}
which is a natural extension of $H_{NM}$ ($\widehat{M} \geq M$),  is exactly $M$.
\end{theorem}
\begin{proof}
By virtue of \eqref{difference}, considering the entries of the first $M$ arrays $[\mathbf{h}^{k_0},\, \dots,\, \mathbf{h}^{k_0+M-2}]$ of $\mathbf{H}_{N\widehat{M}}^{k_0}$ as initial data, we get $\mathbf{h}^{k_0+M-1}$ as a linear combination of these vectors. By changing $M$ into $M+1$ and using $[\mathbf{h}^{k_0+1},\, \dots,\, \mathbf{h}^{k_0+M-1}]$ as initial data for \eqref{difference}, we get $\mathbf{h}^{k_0+M}$ as a linear combination of such vectors and then of $[\mathbf{h}^{k_0},\, \dots,\, \mathbf{h}^{k_0+M-2}]$. Iterating the procedure we obtain that each column vector $[\mathbf{h}^{k_0+M-1},\, \dots,\, \mathbf{h}^{k_0+\widehat{M}-1}]$ is a linear combination of $[\mathbf{h}^{k_0},\, \dots,\, \mathbf{h}^{k_0+M-2}]$, which means that ${\rm{rank} \, \mathbf{H}_{N\widehat{M}}^{k_0}}=M={\rm{rank} \, \mathbf{H}_{NM}^{k_0}}$.
\end{proof}

Our experience suggests that a reliable estimate of $M$ can be obtained  by using a standard MATLAB technique and then applying the numerical method illustrated above.

\section{Numerical Results}\label{tests}
In this section we illustrate the results of an extensive numerical experimentation concerning  various examples, some already considered in the literature and others, to our knowledge, never considered before. 

To ascertain the effectiveness of our method, 
for each example considered, we estimate the relative error for the exponents $f_j$ and the coefficients $c_{js}$ for $j=1,\dots,n$, $s=0,\, \dots,\, m_j-1$, by using the following error estimates
\begin{equation}
e({\bf{f}})= \max_{j=1,\,\dots\,,n} \left| 1-\frac{f_j}{f_j^*} \right|, \quad e({\bf{c}})= \max_{ \substack{j=1,\,\dots\,,n \\ s=0,\,\dots,\,m_j-1}} \left| 1-\frac{c_{js}}{c_{js}^*} \right|
\end{equation}
where $f^*_j$ and $c_{js}^*$ denote the exact values of the parameters.
Moreover, denoting by $[0,\, b]$ the domain of $h(x)$ that mainly interest us, we adopt the following relative error estimate  of the monomial-exponential sum:

\begin{equation}
e({\bf{h}})= \max_{x \in X } \left| 1-\frac{h(x)}{h^*(x)} \right|
\end{equation}
where $X=\{x_i=i\frac{b}{50}, \,  i=1,\,\dots,\, 50 \}$.

In each test function  we assume  $M$ unknown and consider both the case of exact data and the case of noisy data. In the latter case we consider white noise, that is we assume  $$h(k)=\tilde h(k)+\delta e_k, \quad k=k_0,\,\dots,\,k_0+2N-1$$
where $\tilde{h}(k)$ denotes the exact values of the monomial exponential sum , $e_k \in [0,\,1]$ is a random array and $\delta$ is the standard deviation of the sampled data.


All the  computations have been carried out in MATLAB with $\epsilon_{machine}=2.22 \, \cdot \, 10^{-16}$

\subsection{Example 1.} \label{esempio6zeri}

Let us first consider an exponential sum already considered in  \cite{Potts2013}. More precisely, assuming $m_1=m_2= \dots =m_n=1$, we considered $h(x)$ as in \eqref{hx1}  with the following coefficients $c_j$ and zeros $z_j$: 
\begin{equation}
{\bf c} = \left[ \begin{matrix}
1 \\
2 \\
3 \\
4 \\
5 \\
6 
\end{matrix}\right],\quad {\bf z} = \left[ \begin{matrix}
0.9856-0.1628i \\
0.9856+0.1628i \\
0.8976-0.4305i \\
0.8976+0.4305i \\
0.8127-0.5690i \\
0.8127+0.5690i 
\end{matrix}\right].
\end{equation}
Considering data without and  with noisy  and taking $b=50$ we obtain the results reported in Table \ref{Table_6zeros_exact} and in Table \ref{Table_6zeros_noise}, respectively.

\begin{table}[!h]
\centering
\begin{tabular}{c|c|ccc}
$N$ & $\widehat{M}$ & $e(\bf{f})$ & $e(\bf{c})$ & $e(\bf{h})$\\
\hline
6   & 6 &  7.56e-09 & 6.35e-09 & 1.21e-07\\
12  & 10 & 8.63e-12  & 1.31e-11  & 1.10e-10\\
24  & 10 & 8.63e-12  & 8.98e-12  &  2.41e-11\\
36  & 10 & 4.75e-12 & 2.77e-11 &  8.64e-11\\
48  & 10 & 6.77e-13 & 5.42e-12 & 1.08e-11\\
\hline
\end{tabular}
\caption{ Error estimates with exact data for Example 1 \label{Table_6zeros_exact}}
\end{table}

\begin{table}[!h]
\centering
\begin{tabular}{c|c|c|ccc}
$N$ & $\delta$ & $\widehat{M}$ & $e(\bf{f})$ & $e(\bf{c})$ & $e(\bf{h})$\\
\hline
6   & $10^{-9}$ &  6  &  1.73e-03 & 2.37e-03 & 2.57e-02	\\
12  & $10^{-9}$ & 10 & 1.26e-07 & 9.77e-07   & 1.91e-05 \\
24  & $10^{-9}$ & 10 & 6.72e-10 & 4.11e-09   & 3.23e-08 \\
36  & $10^{-9}$ & 10 & 1.29e-10 &3.36e-08 &   2.20e-07 \\
48  & $10^{-9}$ & 10 & 4.06e-10  &4.22e-09    &2.35e-08 \\
\hline
\end{tabular}
\caption{Error estimates with  noisy  data for Example 1 \label{Table_6zeros_noise}}
\end{table}

%

It is worthwhile to note that, in the absence of noise, our method identifies the exact values of $M$, regardless the number of data we consider.
Table \ref{Table_6zeros_noise} shows that, if the data are noisy, as it should be expected, the estimate of $M$ is exact in the  case $N=M$ and overestimated whenever $N>M$. Nevertheless, as this table shows, the identification of both the parameters and the coefficients is very accurate even if $M$ is overestimated by $\widehat{M}$.

Moreover, for an immediate comparison of our results with those obtained by the methods considered in \cite{Potts2013}, we computed coefficients and zeros by using the error estimates proposed there. 
Our results, as Table \ref{Table_6zeros_exactpotts} and Table 4.1 of \cite{Potts2013} show, have the same level of error also when our upper bound estimate of $M$ is rather inaccurate.

\begin{table}[!h]
\centering
\begin{tabular}{c|c|ccccc}
 $N$ & $\widehat{M}$ & $e(\bf{f})$ & $e(\bf{c})$ & $e(\bf{h})$\\
\hline
 6  & 6 & 2.02e-09   & 1.07e-09  & 8.63e-15	\\
 7  & 7 & 5.97e-10 &  4.06e-10  & 9.56e-15 \\
 12 & 8 & 2.31e-12 &  2.18e-12 & 1.84e-13 \\
\hline
\end{tabular}
\caption{Further table of errors for Example 1 \label{Table_6zeros_exactpotts}}
\end{table}

\subsection{Example 2.} \label{esempiosignal}
Let $h(x)$ be the exponential sum expressed as in \eqref{hx1} with $M=n=5$ and characterized by the following coefficients and zeros:
\begin{equation}
{\bf c} = e^{15i} \left[ \begin{matrix}
3.1 \\
9.9 \\
6.0 \\
2.8 \\
17 
\end{matrix}\right],\quad {\bf z} = 2 * 10^{-5} \left[ \begin{matrix}
-208-2 \pi 1379 i \\
-256-2\pi 685 i\\
-197-2\pi 271i \\
-117+2\pi 353 i \\
-808+2\pi 478i 
\end{matrix}\right],
\end{equation} 
already considered in \cite{Potts2011}. The error estimates obtained both in the absence and in presence of noisy data are reported in Tables  \ref{table_5zeros_exact} and \ref{table_5zeros_noise}, respectively.
         
\begin{table}[!h]
\centering
\begin{tabular}{c|ccc}
$N$  & $e(\bf{f})$ & $e(\bf{c})$ & $e(\bf{h})$\\
\hline
5   & 3.44e-03 & 1.09e-02  & 4.68e-05	\\
10  & 3.95e-05 & 1.31e-04  & 3.19e-07 \\
15  & 2.10e-05 & 7.30e-05 & 8.09e-08 \\
20  & 3.39e-06 & 1.21e-05 & 5.49e-09\\
50  & 3.63e-08 & 1.53e-07 &  1.66e-09 \\ 
\hline
\end{tabular}
\caption{Error estimates with exact data for Example 2  \label{table_5zeros_exact}}
\end{table}

\begin{table}[!h]
\centering
\begin{tabular}{c|c|c|ccc}
$N$ & $\delta$ & $\widehat{M}$ & $e(\bf{f})$ & $e(\bf{c})$ & $e(\bf{h})$\\
\hline
5 & $10^{-9}$ & 5 & 2.14e+00 & 9.62e-01 & 2.01e+00\\
10 & $10^{-9}$ & 10  &  8.19e-03 &    2.80e-02    &  1.61e-04\\
15 & $10^{-9}$ & 10 &           9.84e-04  & 3.36e-03  &  6.69e-06 \\
10 & $10^{-9}$ & 10 &        2.00e-04  & 6.11e-04   &  2.95e-07 \\
50 & $10^{-9}$ & 10 &         2.21e-06 &1.15e-05     &  1.25e-08 \\
\hline
\end{tabular}
\caption{Error estimates  with noisy data for Example 2 \label{table_5zeros_noise}}
\end{table}

As already noted in \cite[Table 1]{Potts2011}, both tables show that recovering the parameters and coefficients in this example is more complicated than in the previous one. However, we obtain reliable results also for moderately high values of $N$, unlike what happens in \cite{Potts2011}.    

\subsection{Example 3} \label{esempio5zeridouble}
To test the effectiveness of the method in the case of multiple zeros, first we  modify Example 2 by assuming the first zero to be double. That is we assume that the new $h(x)$ function \eqref{hx1} is now characterized by the vector data 
\begin{equation}
{\bf c} = e^{15i} \left[ \begin{matrix}
3.1 \\
9.9 \\
6.0 \\
2.8 \\
17 
\end{matrix}\right],\quad {\bf z} = 2 * 10^{-5} \left[ \begin{matrix}
-208-2 \pi 1379 i \\
-208-2 \pi 1379 i \\
-197-2\pi 271i \\
-117+2\pi 353 i \\
-808+2\pi 478i 
\end{matrix}\right].
\end{equation} 
We note that our method gives reliable results also in this more complex situation as Tables \ref{table_5zeros_double_exact} and \ref{table_5zeros_double_noise} show.
\begin{table}[!h]
\centering
\begin{tabular}{c|c|ccc}
$N$ & $\widehat{M}$ & $e(\bf{f})$ & $e(\bf{c})$ & $e(\bf{h})$\\
\hline
5   & 5 & 2.86e-03 & 1.99e-01 &  4.48e-03 \\
10  & 10 & 4.56e-05 & 2.32e-02 & 4.46e-04\\
15  & 10 & 2.42e-05 & 1.22e-02 & 1.61e-04\\
20  & 10 & 9.10e-06 & 4.57e-03 & 2.95e-05\\
50  & 10 & 3.80e-06 & 1.57e-03 & 2.11e-04\\
\hline
\end{tabular}
\caption{Error estimates  with exact data for Example 3 \label{table_5zeros_double_exact}}
\end{table}

\begin{table}[!h]
\centering
\begin{tabular}{c|c|c|ccc}
$N$ & $\delta$ & $\widehat{M}$ & $e(\bf{f})$ & $e(\bf{c})$ & $e(\bf{h})$\\
\hline
5   & $10^{-9}$ &  5 & 4.87e+00 & 9.33e+01 & 5.65e+00  \\
10  & $10^{-9}$ & 10 & 2.95e-03 & 2.96e-01 & 5.63e-03\\
15  & $10^{-9}$ & 10 & 5.78e-04 & 1.73e-01 & 2.32e-03\\
20  & $10^{-9}$ & 10 & 1.49e-04 & 7.50e-02 & 4.83e-04\\
50  & $10^{-9}$ & 10 & 9.93e-06 & 4.10e-03 & 5.50e-04 \\
\hline
\end{tabular}
\caption{Error estimates  with  noisy  data for Example 3 \label{table_5zeros_double_noise}}
\end{table}

\subsection{Example 4} \label{esempio5zeri2double}
Let us consider again Example 2 assuming  that the first two zeros are double and the third is simple, that is setting 
\begin{equation}
{\bf c} = e^{15i} \left[ \begin{matrix}
3.1 \\
9.9 \\
6.0 \\
2.8 \\
17 
\end{matrix}\right],\quad {\bf z} = 2 * 10^{-5} \left[ \begin{matrix}
-208-2 \pi 1379 i \\
-208-2 \pi 1379 i \\
-256-2 \pi 685i \\
-256-2 \pi 685i \\
-197-2 \pi 271i \\
\end{matrix}\right].
\end{equation} 

The errors obtained in absence as in presence of noisy are reported in Table \ref{table_5zeros_2double_exact} and \ref{table_5zeros_2double_noise}, respectively. Both tables show that, also in the case where the estimate of $M$ is largely inaccurate, we obtain acceptable results for moderately high values of $N$.

\begin{table}[!h]
\centering
\begin{tabular}{c|c|ccc}
$N$ & $\widehat{M}$ & $e(\bf{f})$ & $e(\bf{c})$ & $e(\bf{h})$\\
\hline
5   & 5  & 2.07e-02 & 1.05e+00 & 1.26e-02 \\
10  & 10 & 3.98e-03 & 2.08e-01 & 1.60e-03\\
15  & 10 & 2.51e-03 & 1.33e-01 & 5.58e-04\\
20  & 10 & 1.22e-03 & 6.47e-02 & 1.13e-04 \\
50  & 10 & 2.54e-04 & 2.20e-02 & 1.23e-03\\
\hline
\end{tabular}
\caption{Error estimates   with exact data for Example 4 \label{table_5zeros_2double_exact}}
\end{table}

\begin{table}[!h]
\centering
\begin{tabular}{c|c|c|ccc}
$N$ & $\delta$ & $\widehat{M}$ & $e(\bf{f})$ & $e(\bf{c})$ & $e(\bf{h})$\\
\hline
5   & $10^{-9}$ & 5 & 5.17e-01 & 8.95+00 & 8.37e-01 \\
10  & $10^{-9}$ & 10 & 3.96e-02 & 5.57e+00 &9.07e-02 \\
15  & $10^{-9}$ & 10 & 1.16e-02 &  9.22e-01 & 3.94e-03\\ 
20  & $10^{-9}$ & 10 & 5.12e-03 & 2.90e-01 & 5.03e-04\\
50  & $10^{-9}$ & 10 & 5.62e-04 &  5.27e-02 & 1.81e-03\\
\hline
\end{tabular}

\caption{Error estimates with  noisy  data for Example 4 \label{table_5zeros_2double_noise}}
\end{table}
      
\subsection{Example 5}\label{esempio6zeridouble}
Let us now return to the first example assuming that the zeros $z_1=0.9856-0.1628i$ and $z_2=0.8976-0.4305i$ are double and the zeros $z_3=0.8127-0.5690i$ and $z_4=0.8127+0.5690i$ are simple. As we can see by our numerical results reported in Table \ref{table_6zeros_double_exact} and in Table \ref{table_6zeros_double_noise},  although two zeros are not simple and $M$, the recovering of the parameters and the sum is still accurate and improves as the number of data increases. 

\begin{table}[!h]
\centering
\begin{tabular}{c|c|ccc}
$N$ & $\widehat{M}$ &  $e(\bf{f})$ & $e(\bf{c})$ & $e(\bf{h})$\\
\hline
6   & 6  & 1.98e-04 & 2.08e-02 & 9.59e-01\\
12  & 10 & 1.73e-05 & 2.57e-03 & 6.67e-06 \\
24  & 10 & 4.08e-06 & 9.48e-04 & 6.62e-01 \\
36  & 10 & 2.79e-06 & 1.65e-03 & 1.31e-04\\
48  & 10 & 2.71e-06 & 2.81e-03 & 2.43e-04 \\
\hline
\end{tabular}
\caption{ Error estimates with exact data for Example 5 \label{table_6zeros_double_exact}}
\end{table}

\begin{table}[!h]
\centering
\begin{tabular}{c|c|c|ccc}
$N$ & $\delta$ & $\widehat{M}$ & $e(\bf{f})$ & $e(\bf{c})$ & $e(\bf{h})$\\
\hline
6   & $10^{-9}$ &  6 & 2.45e-02 & 4.52e-01 &  2.07e+00 \\
12  & $10^{-9}$ & 10 & 8.48e-04 & 9.26e-02 & 4.93e-03\\
24  & $10^{-9}$ & 10 & 6.81e-05 & 2.59e-02 & 1.58e-03 \\
36  & $10^{-9}$ & 10 & 1.88e-05 & 1.28e-02 & 1.18e-03\\
48  & $10^{-9}$ & 10 & 1.03e-05 & 1.26e-02 & 9.28e-04 \\
\hline
\end{tabular}
\caption{Error estimates  with  noisy  data for Example 5  \label{table_6zeros_double_noise}}
\end{table}

\subsection{Example 6} \label{Cerchi} 
In this example we consider the identification of the $\{z_j\}$ and $\{c_j\}$ in the sum 
$$h(x)=\sum_{j=1}^{M}c_j z_j^x.$$
It generalizes the example considered in \cite{Potts2011} where $M=30$. In our numerical results we considered $M=40$ and, as in \cite{Potts2011}, the $c_j$ coefficients 
as random values on $[0,1]$ and the $z_j$ values as equidistant nodes on three circles having radius $r=0.7,0.8,0.9$. The results are reported in Figure \ref{Cerchi}, where the exact nodes are depicted as circles and their recovery by stars on the left for the exact data and on the right for inexact data. The figure shows that the collection of $z_j$ is very accurate in absence of noise and reliable in presence of noise and comparatively more accurate with respect to that one reported in \cite[Figure 1]{Potts2011}.
The error estimates for the coefficients $\{c_j\}$ and $h(x)$ are given in Table \ref{table_cerchi} and \ref{table_cerchi_noise} for exact and noisy data. 

\begin{figure}[!t]
\centering
\includegraphics[scale=0.4]{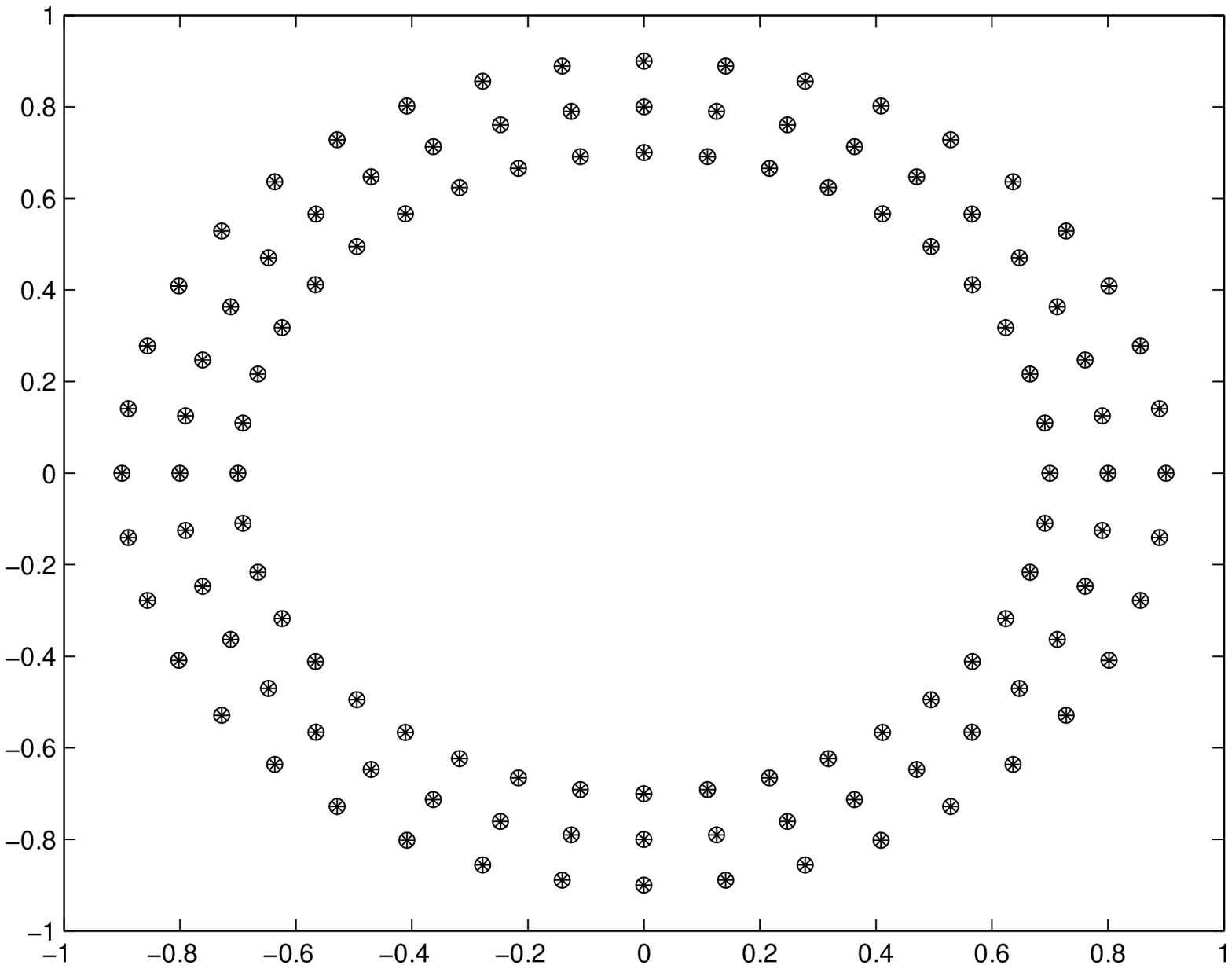} 
\includegraphics[scale=0.4]{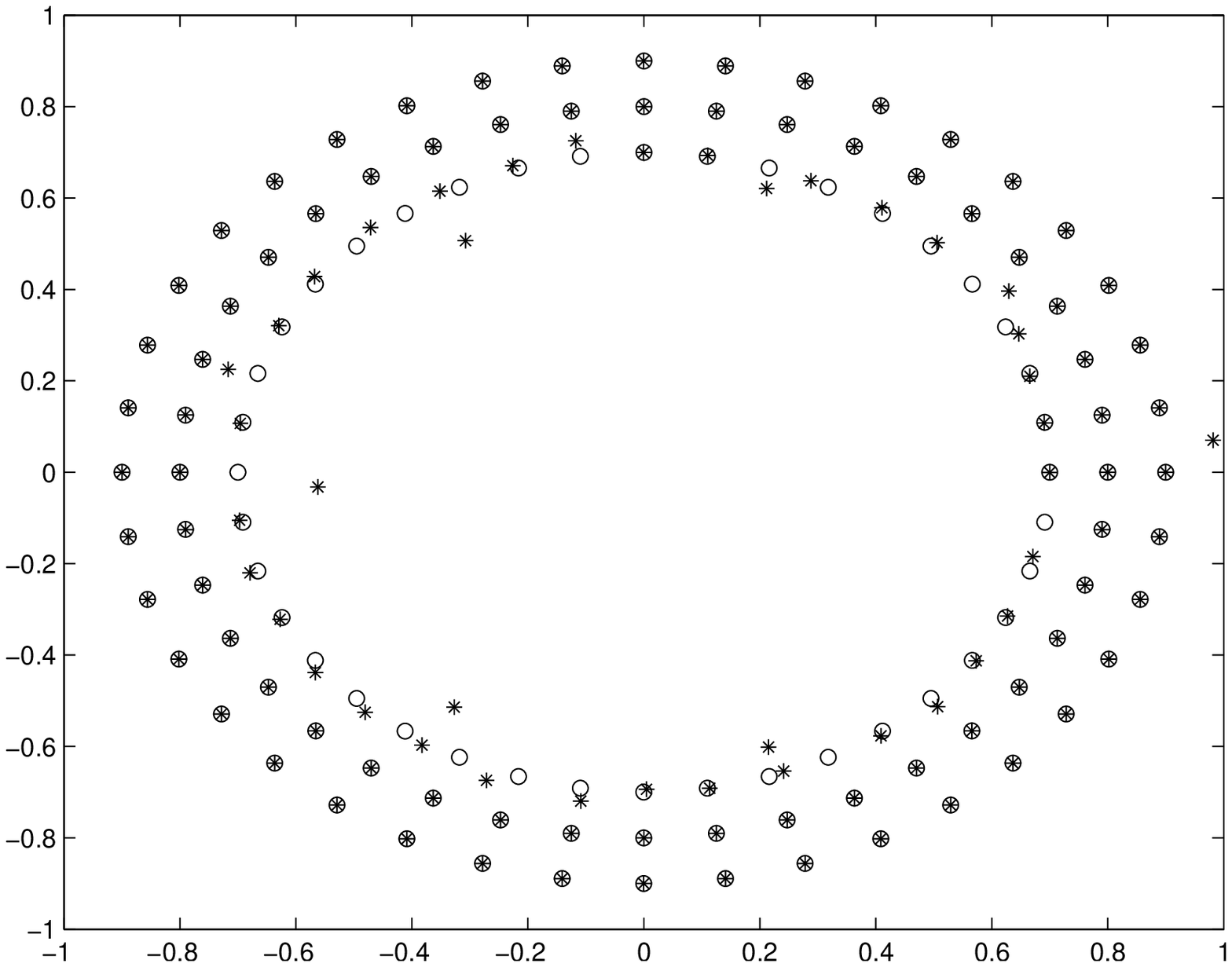} 
\caption{Graphic representation of the nodes of Example 6  for exact data (to the left) and for noisy data with  $\delta=10^{-11}$ (to the right)}
\end{figure}
 
\begin{table}[!h]
\centering
\begin{tabular}{c|c|c|cc}
${\rm{radius}}$ & $N$ & $\widehat{M}$& $e({\bf{c}})$ & $e({\bf{h}})$ \\
\hline 
0.7 & 40  & 40 & 7.012727700027700e-09 & 6.034580367199850e-09\\
0.8 & 40 & 40 &1.215316907382198e-10 & 3.953506430485780e-11 \\
0.9 & 40  & 40 & 1.008568828603852e-11  &  1.766098621871178e-12 \\
\hline 
\end{tabular}
\caption{Error estimates  with exact data for Example 6  \label{table_cerchi}}
\end{table}

\begin{table}[!h]
\centering
\begin{tabular}{c|c|c|cc}
${\rm{radius}}$ & $N$ &  $\widehat{M}$ & $e({\bf{c}})$ & $e({\bf{h}})$ \\
\hline 
0.7 & 40  & 40 & 2.467964763571590e+000 & 2.322316566811376e-002\\
0.8 & 40 & 40 & 5.512014563660308e-003 & 1.484741750804834e-005 \\
0.9 & 40 & 40 & 1.041221887751012e-006  &  1.231439422885003e-007 \\
\hline 
\end{tabular}
\caption{Error estimates with noisy data for Example 6 \label{table_cerchi_noise}}
\end{table}

\subsection{An application to non-linear partial differential equations of integrable type}\label{applications}
An extensive area where  effective methods for parameter identification  in sums of monomial-exponential functions can be very useful is represented by the important class of non-linear partial differential equations (NPDEs) of integrable type. In this context  the non-linear Schr\"odinger equation (NLS), which governs the signal transmission in optical fibers \cite{HasMat2003}, plays a special role.  

The main characteristic of this class is the fact that any initial value problem associated to an NPDE of integrable kind can theoretically be solved by using the inverse scattering transform technique (IST). This technique is primarily based on the solution of a direct scattering problem  and then on the solution of an inverse scattering problem, starting from the spectral data previously obtained by time evolution. From the numerical point of view, the first one is actually the most challenging, at least for the NLS, since the second one can be solved by using the numerical method proposed in \cite{ArRoSe2011}.   

The numerical solution of the direct scattering problem for the NLS is primarily  based on the computation of the initial Marchenko kernels from the left and from the right, respectively \cite{VanDerMee2013}.

These kernels, whenever the solution of the NLS is represented by one soliton as well as by a multisoliton (the so-called reflectionless case), can be represented as follows

\begin{align}
\Omega_\ell(x)=\sum_{j=1}^n e^{-a_j x} \sum_{s=0}^{m_j-1} (\Gamma_\ell)_{js} \frac{x^s}{s!}, \quad x \in \RR^+, \label{Omegal}\\
\Omega_r(x)=\sum_{j=1}^n e^{a_j x} \sum_{s=0}^{m_j-1} (\Gamma_r)_{js} \frac{x^s}{s!}, \quad x \in \RR^-, \label{Omegar}
\end{align}
where $0^0 \equiv 1$ and $a_j$ are complex or real parameters with $Re(a_j)>0$. 

The application of our method to $\Omega_\ell$ allows us to estimate  $\{n,m_j,(\Gamma_\ell)_{js} \}$, knowing $\Omega_\ell$ in $2N$ ($N>M$) positive integer points, and then, to recover  $(\Gamma_r)_{js}$ by solving, in the least squares sense, a linear system of order $N \times M$, given  $\Omega_r$ in $2N$ ($N>M$)  negative integer nodes. The same results can of course be obtained by applying first the method to $\Omega_r(x)$ to identify $\{n,m_j,(\Gamma_r)_{js}\}$ and then to $\Omega_\ell(x)$ to identify $(\Gamma_\ell)_{js}$.

In Tables \ref{table_soliton_simple} and \ref{table_soliton_multiple} we give the error estimates that we obtain in the identification of $\Omega_\ell$ parameters and coefficients  in the following two cases (representative of four-solitons with 4 simple bound states and with a double an two simple bound states):
\begin{enumerate}
\item[(a)] $n=4$, \quad  $m_1=\ldots=m_4=1$, \\ ${\bf{a}}=\frac{1}{10}[1+7i,1.2+3i,1.4+6i,3+1.6i]$ \quad and \quad ${\bf{\Gamma_\ell}}=[1+i, 2+i, 3+i, 4+i]$; \\
\item[(b)] $n=3$, \quad $m_1=2$,\quad  $m_2=m_3=1$, \\ ${\bf{a}}=\frac{1}{10}[1+7i,1.4+6i,3+1.6i]$ \quad and  ${\bf{\Gamma_\ell}}=[1+i, 2+i, 3+i, 4+i]$.
\end{enumerate} 
In both cases we considered $[0,\,5]$ as interval of effective interest and then we assumed $b=5$.

\begin{table}[!h]
\centering
\begin{tabular}{c|c|c|ccc}
$N$ & $\delta$ &$\widehat{M}$ & $e(\bf{f})$ & $e(\bf{c})$ & $e(\bf{h})$\\
4 & 0 & 4 & 1.02e-10 & 1.28e-09  & 4.76e-15 \\
8 & 0 & 7 & 1.33e-11 & 1.58e-10 & 1.08e-14\\
16 & 0 & 7 & 9.90e-14 & 1.11e-12 &  3.24e-15 \\
32 & 0 & 7 & 5.86e-13 & 7.15e-12 & 3.44e-15\\
64 & 0 & 7 & 7.33e-13 & 9.63e-12 & 4.43e-15\\
\hline
4 & $10^{-9}$ & 4 & 7.13e-05 & 9.83e-04 & 2.48e-09  \\
8 & $10^{-9}$ & 7 & 2.70e-07   &  3.44e-06  &   2.85e-10\\
16 &  $10^{-9}$ & 7 & 8.14e-08  &   1.01e-06  &   2.02e-09\\
32 &  $10^{-9}$ & 7 & 5.79e-09  &   9.85e-08  &   3.69e-10\\
64 & $10^{-9}$ & 7 & 2.44e-08  &   3.82e-07  &   4.83e-10\\
\hline 
4 & $10^{-7}$ & 4 & 4.56e-03   &  6.41e-02  &   9.25e-08\\
8 & $10^{-7}$ & 7 & 3.32e-05  &   4.63e-04 &    4.32e-08\\
16 &  $10^{-7}$ & 7 & 7.33e-06  &   1.17e-04  &   1.21e-07\\
32 & $10^{-7}$ & 7 & 1.13e-06   &  1.89e-05  &   6.06e-08\\
64 & $10^{-7}$ & 7 &  1.79e-06  &   2.43e-05  &   4.60e-08\\
\hline
\end{tabular}
\caption{Error estimates in the multisolitons case $(a)$ \label{table_soliton_simple}}
\end{table}      
 
\begin{table}[!h]
\centering
\begin{tabular}{c|c|c|ccc}
$N$ & $\delta$ &$\widehat{M}$ & $e(\bf{f})$ & $e(\bf{c})$ & $e(\bf{h})$\\
4 & 0 & 4  & 5.13e-06  & 5.43e-04 & 4.90e-08   \\
8 & 0 & 7  & 1.49e-06   & 1.76e-04  &   1.66e-07 \\
16 & 0 & 7 & 4.85e-07 &  7.14e-05 &  2.63e-07 \\
32 & 0 & 7 & 3.18e-07 & 5.34e-05 & 3.06e-07 \\
64 & 0 & 7 & 3.38e-07 & 5.70e-05 & 3.29e-07\\
\hline
4 & $10^{-9}$ & 4 & 3.17e-04   &  5.38e-02  &   3.09e-04\\
8 & $10^{-9}$ & 7 &  2.45e-04   &  2.91e-02  &   2.73e-05\\
16 &  $10^{-9}$ & 7 & 4.04e-05   &  5.96e-03  &   2.20e-05 \\
32 &  $10^{-9}$ & 7 &  2.49e-05   &  4.19e-03  &   2.40e-05 \\
64 &  $10^{-9}$ & 7 & 4.02e-05  &   6.78e-03   &  3.92e-05\\
\hline 
4 & $10^{-7}$ & 4 & 2.44e-02  &   2.25e+00  &   2.17e-04 \\
8 & $10^{-7}$ & 7 & 3.44e-03   &  2.95e-01  &   3.82e-04 \\
16 & $10^{-7}$ & 7 & 8.83e-04  &   1.29e-01  &   4.81e-04\\
32  &  $10^{-7}$ & 7 & 3.41e-04  &   5.76e-02 &    3.28e-04 \\
64 &  $10^{-7}$ & 7 & 3.17e-04  &   5.38e-02  &   3.09e-04  \\
\hline
\end{tabular}
\caption{Error estimates in the multisolitons case $(b)$ \label{table_soliton_multiple} }
\end{table}   

Table \ref{table_soliton_simple} highlights that the identification of parameters and coefficients is at all satisfactory in case (a). Table \ref{table_soliton_multiple} 
shows that the situation is more complex if there are multiple bound states (case (b)) as people working in the NPDEs area of integrable type know well. Nevertheless, the results that we obtain are very good in the absence of noise and reliable in the presence of noise, also when $M$ is not known in advance.

\section{Conclusions}\label{conclusions}
The results of our extensive experimentation show that the method allows us to estimate with good precision the parameters and the coefficients of a monomial-exponential sum, even if its number of terms it is not known, provided it is a reasonable overestimation. The method furnishes very accurate results in the absence of noise and acceptable results in the presence of moderately high level of noise, whenever a relatively high number of data, with respect to the number of parameters and coefficients to identify, is available. Finally, we point out that the method, without any algorithmic variant, gives good results even if some parameters correspond to multiple zeros of the polynomial of Prony.  

\section*{Acknowledgments}
The research was partially supported by the Italian Ministery of Education and Research (MIUR) under PRIN grant No. 2006017542-003, by INDAM, and by Autonomous Region of Sardinia under grant L.R.7/2007 ``Promozione della Regione Scientifica e della Innovazione Tecnologica in Sardegna''.

\bibliographystyle{plain}
\bibliography{biblio}
\end{document}